\documentclass[12pt, twoside, leqno]{hjm2}

\usepackage{amssymb,amsmath,amsthm,amsfonts,verbatim,graphicx,mathtools,mathrsfs,commath,float,tikz-cd,setspace}
\usepackage{pifont}

\usepackage{tikz,wrapfig}
\usepackage{enumitem}
\usepackage[T1]{fontenc}
\usepackage[hidelinks]{hyperref}

\newtheorem*{ut}{Theorems}

\newtheorem*{ut4}{Theorem III}
\newtheorem*{uts}{Theorem}

\newtheorem{ul}{Lemma}
\newtheorem{uq}{Question}
\newtheorem{up}{Proposition}

\newtheorem{ucl}[subsection]{Claim}

\theoremstyle{definition}

\newtheorem*{ur}{Remark}

\usepackage[justification=centering]{caption}
\DeclareMathOperator{\diam}{diam}

\DeclareMathOperator{\ran}{ran}

\author{David Sumner Lipham}

\email{dsl0003@auburn.edu}

\numberwithin{equation}{section}

\begin{document}

\title[One-to-one composant mappings]{One-to-one composant mappings\\  
of $[0,\infty)$ and $(-\infty,\infty)$}



\maketitle

\renewcommand{\thefootnote}{}



\renewcommand{\thefootnote}{\arabic{footnote}}
\setcounter{footnote}{0}

\vspace{-15mm}

\begin{abstract}%

Knaster continua and solenoids are well-known examples of indecomposable continua whose composants (maximal arcwise-connected subsets) are one-to-one images of lines. We show that essentially   all non-trivial  one-to-one  composant images of (half-)lines are indecomposable.  And if $f$ is a one-to-one mapping of $[0,\infty)$ or $(-\infty,\infty)$, then there is an indecomposable continuum of which  $X:=\ran(f)$  is a composant if and only if $f$ maps  all final or initial segments densely and every non-closed sequence of arcs in $X$ has a convergent subsequence in the hyperspace $K(X)\cup \{X\}$. 
Accompanying the proofs are illustrations and examples.
\end{abstract}

\

\section{Introduction}
Throughout, $[0,\infty)$ denotes the half-line and $(-\infty,\infty)$ denotes the entire real line.  Every mapping is assumed to be continuous; by  \textit{image} we shall always mean  \textit{continuous image}.\footnote{Being a one-to-one  image  of  $[0,\infty)$ is the same as being the union of a strictly increasing sequence of arcs which share a common endpoint. Among locally connected, locally compact spaces, there are only $3$ such images, and there are only $5$ such images of $(-\infty,\infty)$ -- Lelek \& McAuley \cite{lm} and Nadler \cite{nad}.  Much is also known about  other types of images:  compact \cite{nad3,beck}, confluent \cite{nad2}, aposyndetic \cite{jon2}, uniquely arcwise-connected \cite{roma}, hereditarily unicoherent \cite{over}. }   All images of the (half-)line are assumed  to be metrizable, and by a \textit{continuum} (plural form \textit{continua}) we mean a connected compact metrizable space.  An \textit{arc} is a homeomorphic copy of the interval $[0,1]$.

A continuum $Y$ is \textit{decomposable} if there are two subcontinua $H,K\subsetneq Y$ such that $Y=H\cup K$; otherwise $Y$ is \textit{indecomposable}. We shall say, more generally, that  a connected space $X$ is \textit{indecomposable} if $X$ cannot be written as the union of two proper closed connected subsets.  Equivalently, $X$ is indecomposable if $X$ is the only closed connected subset of $X$ with non-void interior.

If $Y$ is a continuum and $x\in Y$,  then \textit{$X$ is the composant of $x$ in $Y$} means  that $$X=\bigcup \{K\subsetneq Y:K \text{ is a continuum and }x\in K\}.$$   More generally, \textit{$X$ is a composant of $Y$} if  there exists $x\in X$ such that $X$ is the composant of $x$ in $Y$.

Given a continuum $Y$, a line $\ell\in \{[0,\infty),(-\infty,\infty)\}$, and a mapping $f:\ell\to Y$, one easily sees that   $\ran(f):=f[\ell]$ (the \textit{range of $f$}) is contained in a composant of $Y$.  The goal of this paper is to describe all one-to-one  images of (half)-lines which are \textit{homeomorphic to}  composants of continua. Theorem I classifies  all decomposable composant images, while  Theorem II provides an internal characterization ``composant-ness'' which is independent of any particular embedding. 


\begin{ut}Let $f:\ell\to X$ be a one-to-one mapping of $\ell\in \{[0,\infty),(-\infty,\infty)\}$ onto a metrizable space $X$. 

\noindent \textnormal{\textbf{I.}} There is a decomposable continuum of which $X$ is a composant if and only if

\begin{enumerate}
\item[$(\textnormal A)$]$X$ is compact  $($this is the only possibility 
  if $\ell=(-\infty,\infty)$$)$;
\item[$(\textnormal B)$] $X\simeq [0,\infty)$, i.e. $f$ is a homeomorphism; or
\item[$(\textnormal C)$]  $\exists\; s\in (0,\infty)$ such that $f[0,s)$ is open in $X$ and $f[s,\infty)$ is an indecomposable composant.
\end{enumerate}
\hspace{\parindent}\begin{minipage}[t]{12.3cm}
Moreover,  if $X$ is neither compact nor equal to the half-line, and  $Y$ is a decomposable continuum of which $X$ is a composant, then there exists $s\in (0,\infty)$ such that $Y\setminus f[0,s)$ is an indecomposable continuum of which $f[s,\infty)$ is a composant.\end{minipage}

\vspace{2mm}

\noindent \textnormal{\textbf{II.}} There is an indecomposable continuum of which  $X$ is a composant  if and only if 
\begin{enumerate}
\item[$(1)$]  $\overline{f[n,\infty)}=X$ or $\overline{f[(-\infty,-n]\cap \ell]}=X$ for every $n<\omega$; and
\item[$(2)$]  $\overline{\bigcup \{A_n:n<\omega\}}\in K(X)\cup \{X\}$ 
 for every sequence of arcs $(A_n)\in [K(X)]^\omega$ such that $c(X)\cap \prod \{A_n:n<\omega\}\neq\varnothing$.
\end{enumerate}
\hspace{\parindent}
\end{ut}

In condition (2) of Theorem II, $K(X)$  is the set of non-empty compact subsets of $X$, and $c(X)=\{(x_n)\in X^\omega:(\exists x\in X)(x_n\to x\text{  as }n\to\infty)\}$  is the set of convergent point sequences in $X$. 

In Section 7 we will prove two fairly general lemmas to obtain the following. 
 
 \begin{ut4}For every linear composant $X$ there is a continuum $Y\subseteq [0,1]^3$ such that $\dim(Y)=1$ and  $X$ is a composant of $Y$.\end{ut4}
  
  \noindent Here and elsewhere, the term \textit{linear} is used to indicate a space which is a one-to-one image of the line or half-line.
 
 We conclude in Sections 8 and 9 with  several relevant examples and a list of important questions about  composant embeddings,  chainability, and indecomposability of   first category  plane images.



\section{Recurrence}

Suppose  $f$ is a mapping of $\ell \in \{[0,\infty),(-\infty,\infty)\}$. Let $X=\ran(f)$.  If $\ell=[0,\infty)$, then we say $f$ is \textit{recurrent} if $\overline{f[n,\infty)}=X$ for each $n<\omega$. If $\ell=(-\infty,\infty)$, then $f$ is:
\begin{itemize}
\item \textit{positively-recurrent} if $\overline{f[n,\infty)}=X$ for each $n<\omega$;
\item \textit{negatively-recurrent} if $\overline{f(-\infty,-n]}=X$ for each $n<\omega$;
\item \textit{recurrent} if $f$ is positively \tt{or} \textnormal{negatively recurrent; and}
\item \textnormal{\textit{bi-recurrent} if $f$ is both positively }\tt{and}\textnormal{ negatively recurrent.}
\end{itemize}
\begin{ur}According to definitions, condition (1) in Theorem II says ``$f$ is recurrent''.\end{ur}

\begin{up} If $f$ is one-to-one and recurrent, then $X$ is of the first category of Baire. \end{up}

\begin{proof}By the hypotheses, if $\ell=[0,\infty)$ then each $f[0,n]$ is nowhere dense in $X$. Likewise, if $\ell=(-\infty,\infty)$ then each $f[-n,n]$ is nowhere dense in $X$.
\end{proof}


\begin{up} If  $X$ is non-degenerate and indecomposable, then $f$ is recurrent.\end{up}

\begin{proof}By contraposition. Suppose $f$ is not recurrent. 

 \textit{Case 1:} $\ell=[0,\infty)$.   Then there exists $n<\omega$ such that  $\overline{f[n,\infty)}\neq X$. If $f[0,n]\neq X$ then $X$ is the union of the two proper closed connected sets $f[0,n]$ and $\overline{f[n,\infty)}$, whence  $X$ is decomposable. On the other hand, if $f[0,n]=X$ then $X$ is locally connected by the Hahn-Mazurkiewicz Theorem.  Then either $X$ is either degenerate  or  decomposable.

 \textit{Case 2:} $\ell=(-\infty,\infty)$. Then there exists $n<\omega$ such that $\overline{(-\infty,-n]}\neq X\neq \overline{[n,\infty)}$. The goal is to show $X$ is decomposable, so we may assume each of the connected sets $\overline{f(-\infty,-n]}$ and $\overline{f[n,\infty)}$ is nowhere dense in $X$.  Then $f[-n,n]$ has non-void interior in $X$, so if $f[-n,n]\neq X$ then $X$ is automatically decomposable.  Otherwise  $X$ is  locally connected and $X$ is either degenerate or decomposable. \end{proof}
 


 \begin{up}\label{b}Suppose $f$ is one-to-one and  $Y$ is a continuum of which $X$ is a composant. Then $f$ is recurrent if and only if $Y$ is indecomposable.\end{up}

\begin{proof}  In a decomposable continuum every composant has non-void  interior.  So by Proposition 1 and the Baire Category Theorem,  if $f$ is recurrent then it must be that $Y$ is indecomposable.  Conversely, if $Y$ is indecomposable then $X$ is indecomposable by  $\overline{\overline X} =Y$ (composants are dense).  Then $f$ is recurrent by Proposition 2.   \end{proof}


\section{Proof of Theorem I}

Let $f:\ell\to X$ be a one-to-one mapping of $\ell\in \{[0,\infty),(-\infty,\infty)\}$ onto  $X$. 
 \smallskip

\noindent \textit{Any of (A) through (C) is sufficient:}
 \smallskip
 
 If (A), then $X$ is a decomposable continuum (Propositions 1 \& 2) with composant equal to itself.  If (B),  then $X$ is the composant of $f(0)$ in the one-point compactification of $X$,  just as $[0,1)$ is the composant of $0$ in the unit interval $[0,1]$.\footnote{Moreover, if $Y$ is any compactification of $X\simeq [0,\infty)$, then  $X$ is the composant of $f(0)$ in $Y$.} If (C), and $Y$ is an indecomposable continuum of which $f[s,\infty)$ is a composant, then $X$ is the composant of $f(0)$ in the decomposable continuum $f[0,s) \cup Y$.

\medskip

\noindent\textit{One of (A) through (C) is necessary:}
 \smallskip
 
 Suppose $Y$ is a decomposable continuum of which $X$ is a composant. 
 \smallskip

 \textit{Case 1:}  $\ell=(-\infty,\infty)$. 
  \smallskip
  
 We show (A). Well, suppose for a contradiction $X$ is non-compact. Then at least one of $\overline{f(-\infty,0]}$ and  $\overline{f[0,\infty)}$ is non-compact.  Without loss of generality,  $\overline{f[0,\infty)}$ is non-compact.  Let $r\in(-\infty,\infty)$ be such that $X$ is the composant of $f(r)$, and let $n<\omega$.  We have $\overline {f[r,\infty)}=X$ by maximality of $X$, so $f(-\infty,r)\subseteq \overline{f[n,\infty)}$. Then $f(r)\in \overline{f[n,\infty)}$.  As before,   $\overline{f[n,\infty)}=X$. Since $n<\omega$ was arbitrary, $f$ is (positively-)recurrent. By Proposition 3, this contradicts decomposability of $Y$. 

 \smallskip
 \textit{Case 2:}  $\ell=[0,\infty)$.  
  \smallskip
  
 Suppose  neither (A) nor (B) holds. We show (C). 

\begin{ucl}$X$ is the composant of $f(0)$ in $Y$. \end{ucl}

\begin{proof}[Poof of Claim 3.1]Let $t\in [0,\infty)$ be such that $X$ is the composant of $f(t)$ in $Y$.  Let $P$ be the composant of $f(0)$ in $Y$. Apparently $X\subseteq P$ because for each $x\in X$   the arc $f[0,f^{-1}(x)]$ is a proper subset of $Y$ (by $\neg$(A)). Now let  $y\in P$.  There is a continuum $K\subsetneq Y$ with $\{f(0),y\}\subseteq K$.  If $f[0,t]\cup K\neq Y$, then clearly $y\in X$. If $f[0,t]\cup K=Y$, then $f(t,\infty)\subseteq K$, whence $f(t)\in K$.  Again $y\in X$. Thus $P\subseteq X$. Combining both inclusions, we have $P=X$.\end{proof}

\begin{ucl}There exists $t>0$ such that $f[0,t)$ is open in $X$. \end{ucl}

\begin{proof}[Poof of Claim 3.2]For otherwise $f(0)\in \overline{f[n,\infty)}$ for every $n<\omega$.  Let $n<\omega$ such that $\overline{f[n,\infty)}\neq X$ (Proposition 3). Then $\overline{\overline{f[n,\infty)}}$ is a proper subcontinuum of $Y$ containing $f(0)$ and meeting $Y\setminus X$ (by $\neg$(A)).  In light of Claim 3.1, this contradicts maximality of $X$.\end{proof}

Let $s=\sup\{t\in (0,\infty):f[0,t)\text{ is open in }X\}$. Observe that $s>0$ by Claim 3.2, and $s<\infty$ by $\neg$(B).  Also,  $f[0,s)=\bigcup\{f[0,t):t\in (0,\infty)\text{ and }f[0,t)\text{ is open in }X\}$ is open in $X$, and  $f[s,\infty)$ is the composant of $f(s)$ in $\overline{\overline{f[s,\infty)}}$.  

\begin{ucl}  $\overline{\overline{f[s,\infty)}}$ is indecomposable. \end{ucl}

\begin{proof}[Poof of Claim 3.3]By Proposition 3, it suffices to show $f\restriction [s,\infty)$ is recurrent. Suppose to the contrary that  there exists $m<\omega$ such that $f[s+m,\infty)$ is not dense in $f[s,\infty)$. Then $f(s,s+m)$ has non-empty interior in $Y$. By definition of $s$ we have $s\in \overline{f[s+m,\infty)}$. Then $\overline{\overline{f[s+m,\infty)}}$ is a proper subcontinuum of $Y$ containing $f(s)$ and meeting $Y\setminus X$ (by $\neg$(A)).  Then $f[0,s]\cup \overline{\overline{f[s+m,\infty)}}$ is a proper subcontinuum of $Y$ that violates maximality of $X$. \end{proof}

 This concludes the proof of Theorem I.

\section{Proof of Theorem II (Necessity)}
\noindent 

\noindent Suppose $Y$ is an indecomposable continuum of which $X$ is a composant. Then (1) is true by Proposition 3.  Towards proving (2), let $(A_n)\in [K(X)]^\omega$ such that $c(X)\cap \prod \{A_n:n<\omega\}\neq\varnothing$. Let $x$ be the limit point of an element of $c(X)\cap \prod \{A_n:n<\omega\}$. Supposing $\overline{\bigcup \{A_n:n<\omega\}}\notin K(X)$, there exists $y\in\overline{\overline{\bigcup \{A_n:n<\omega\}}}\setminus X$.  Let $K$ be the component of $y$ in $\overline{\overline{\bigcup \{A_n:n<\omega\}}}$.  Observe that $x\in K$.   The composants of an indecomposable continuum are pairwise disjoint, so $X$ is the composant of $x$ in $Y$. Thus  $K=Y$.  It follows that $\overline{\bigcup \{A_n:n<\omega\}}=X$.



\section{Arcs in $\ran(f)$}

Before proving the opposite direction in Theorem II, we need two more propositions (this section) and a key lemma (next section).  

Proposition 4 is  classic.  

\begin{up}\label{c}Let $f:[0,\infty)\to X$ be a continuous bijection onto a non-compact space $X$. For any continuum $K\in K(X)$ there are two numbers $a,b\in [0,\infty)$ such that $a\leq b$ and  $f[a,b]=K$.  In particular, $K$ is a point or an arc.\end{up}

\begin{proof}No tail of $[0,\infty)$ maps into $K$.  For if $f[n,\infty)\subseteq K$, then $X=f[0,n]\cup K$ is compact. So there is an unbounded sequence of numbers $r_0<r_1<...$ in $(0,\infty)\setminus f^{-1}[K]$. Since $[0,\infty)\setminus f^{-1}[K]$ is open, there are two additional sequences $(a_k)$ and $(b_k)$ such that $0=a_0<b_k<r_k<a_{k+1}<b_{k+1}$ for each $k<\omega$ and $(b_k,a_{k+1})\cap f^{-1}[K]=\varnothing$. Then $K$ is covered by the  disjoint compacta $f[a_k,b_k]$, $k<\omega$. By Sierp\-i\'n\-ski's Theorem (stated in Section 6), there exists $k^*$ such that $K\subseteq f[a_{k^*},b_{k^*}]$. Since $f\restriction [a_{k^*},b_{k^*}]$ is a homeomorphism, we can find the desired numbers $a\leq b$ in  $[a_{k^*},b_{k^*}]$. If $a=b$ then $K$ is a point; otherwise $K$ is an arc.
\end{proof}

By nearly identical arguments, Proposition 4 is also true when $\ell=(-\infty,\infty)$ and $f$ is bi-recurrent.

\begin{ur}Let  $f:\ell\to X$ be a continuous bijection onto a composant space $X$, where $\ell\in \{[0,\infty),(-\infty,\infty)\}$.  If $\ell=[0,\infty)$ and $f$ is recurrent, or $\ell=(-\infty,\infty)$ and $f$ is bi-recurrent,  then the following are true.  By the composant property every proper closed connected subset of $X$ is compact. So by Proposition 4, $f$ is confluent and every non-degenerate proper closed connected subset of $X$ is an arc.  In particular, $X$ is hereditarily unicoherent.   Since neither end of $X$ terminates to form a circle, $X$ is also uniquely arcwise-connected.\end{ur}


A mapping $f$ of the line or half-line $\ell$ is \textit{arc-complete} provided for every three sequences $a,b,c\in \ell^{\hspace{.5mm}\omega}$ such that $a_n<b_n$ and  $c_n\in [a_n, b_n]$, if $f(c)$ converges in $X:=\ran(f)$ 
then $\overline{\bigcup\{f[a_n,b_n]:n<\omega\}}$ is compact or equal to $X$.

We now give a subsequence criterion for arc-completeness. The topology on $K(X)$ is the Vietoris topology (equals the  topology generated by a Hausdorff metric).

\begin{up}\label{d}Let $f:[0,\infty)\to X$ be a continuous bijection. The following are equivalent.
\begin{enumerate}[label=\textnormal{(\roman*)}]
\item $f$ is arc-complete;

\item  condition (2) in Theorem II;

\item  for every sequence of arcs $(A_n)\in [K(X)]^\omega$, if $c(X)\cap \prod\{A_n:n<\omega\}\neq\varnothing$ and $\overline{\bigcup\{A_n:n<\omega\}}\neq X$ then a subsequence of $(A_n)$ converges in $K(X)$.  
 \end{enumerate}
\end{up}

\begin{proof}(i)$\Leftrightarrow$(ii) is an immediate consequence of Proposition 4. 

 We will now prove (ii)$\Leftrightarrow$(iii).  For this purpose, let $Y$ be any metric compactum in which $X$ is densely embedded.

Suppose (ii). Let $(A_n)\in [K(X)]^\omega$ be such that $c(X)\cap \prod\{A_n:n<\omega\}\neq\varnothing$ and $\overline{\bigcup\{A_n:n<\omega\}}\neq X$. Then  $\overline {\bigcup \{A_{n}:n<\omega\}}$ is compact by hypothesis. By compactness of $K(Y)$, a  subsequence $(A_{n_k})$ converges to a point $K\in K(Y)$. Then $$K\subseteq \overline{\overline{\bigcup\{A_{n_k}:k<\omega\}}}= \overline{\bigcup\{A_{n_k}:k<\omega\}},$$ whence $K\in K(X)$.  This proves (iii).

Conversely, suppose (iii). Let $(A_n)\in [K(X)]^\omega$ be such that $c(X)\cap \prod\{A_n:n<\omega\}\neq\varnothing$ and $\overline{\bigcup\{A_n:n<\omega\}}\notin\{X\}$. We show $\overline{\bigcup\{A_n:n<\omega\}}\in K(X)$.  To that end, let $y\in \overline{\overline{\bigcup\{A_n:n<\omega\}}}$ and show $y\in X$.  There exists $(y_k)\in[\bigcup\{A_n:n<\omega\}]^\omega$ such that $y_k\to y$. For each $k<\omega$, let $n_k$ be such that $y_k\in A_{n_k}$. By compactness of each $A_n$, we may assume that $\{n_k:k<\omega\}$ is infinite and in strictly increasing order.  Applying the hypothesis to $(A_{n_k})$, we find that a subsequence of $(A_{n_k})$ converges to a point  $K\in K(X)$. Then $y\in K\subseteq X$.  Since  $y$ was arbitrary, we have  $$\overline{\bigcup\{A_n:n<\omega\}}=\overline{\overline{\bigcup\{A_n:n<\omega\}}}\in K(X).$$ This proves (ii).\end{proof}

\noindent Proposition 5 is also true with $(-\infty,\infty)$ in the place of $[0,\infty)$, but proving (i)$\Rightarrow $(ii) requires some care. Suppose (i), and assume $X$ is non-compact. Let $(A_n)\in [K(X)]^\omega$. If for every $n<\omega$, $(-\infty,\infty)\setminus f^{-1}[A_n]$ is unbounded in the positive and negative directions, then the pre-image of each arc is a closed and bounded interval (Proposition 4 arguments), and (ii) follows. If, on the other hand,  an initial or final segment of $(-\infty,\infty)$ maps into an arc $A_m$, then every other $A_n$ ($n\neq m$) has a pre-image $[a_n,b_n]$.  Applying (i) to these intervals will show (ii).%




\newpage
\section{Zero-dimensional collections of arcs}

\begin{wrapfigure}[14]{R}{0.45\textwidth}
\centering
 \includegraphics[scale=.38]{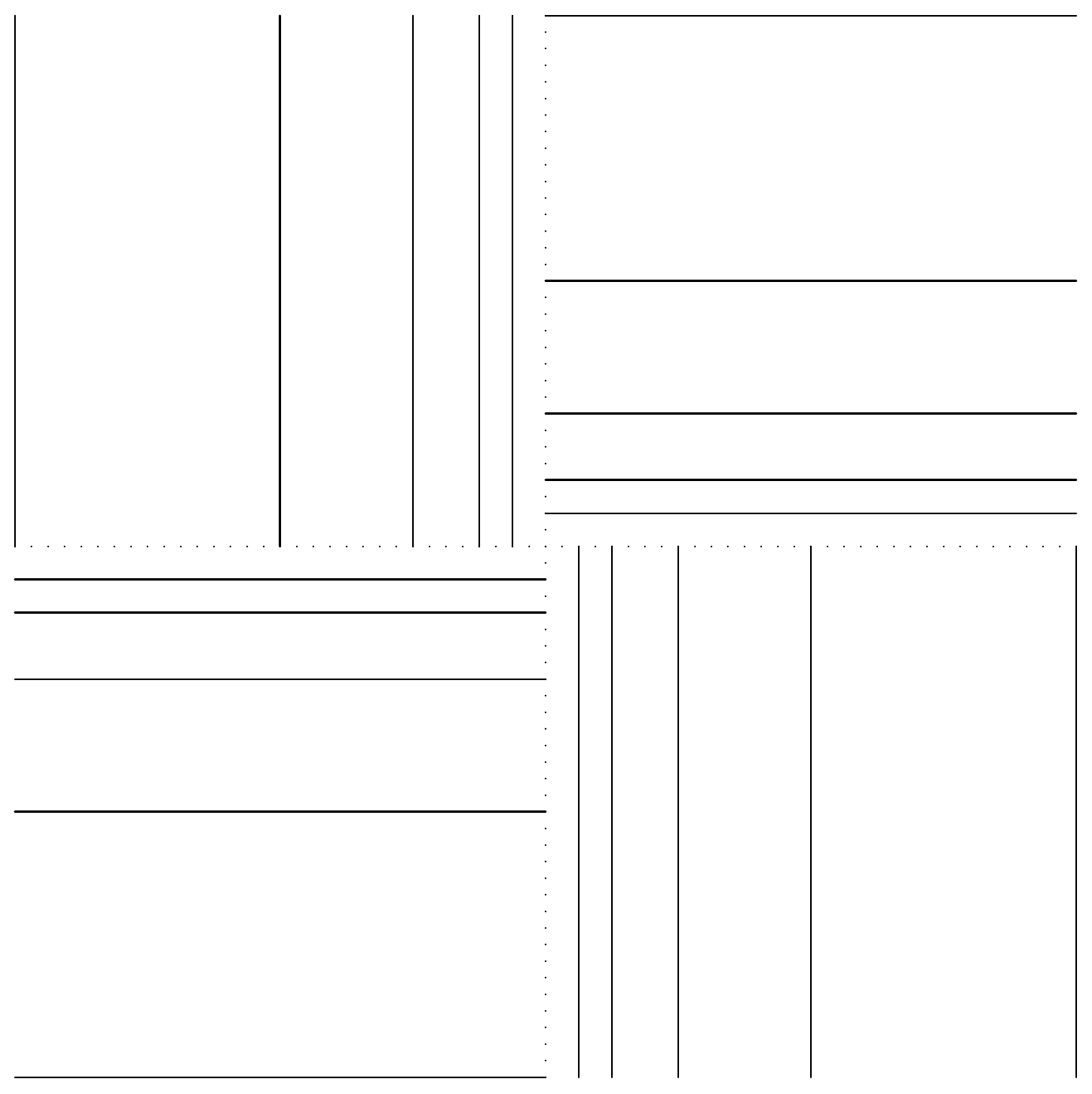} 
  \centering \caption{$\mathfrak C_\omega$}
  \label{fiffe}
\end{wrapfigure}

The result in this section is  based on `The Sierp\-i\'n\-ski Theorem' -- \cite{eng} 6.1.27.

\begin{uts}[Sierp\-i\'n\-ski]If the continuum $X$ has a  countable cover $\{X_i:i<\omega\}$ by pairwise disjoint closed subsets, then at most one of the sets $X_i$ is non-empty.\end{uts}

That statement  is false with `connected space' in the place of `continuum'; Figure 1 shows a connected space $\mathfrak C_\omega$ which is the union of $\omega$-many disjoint arcs;  $\mathfrak C_\omega=\bigcup \{C_n:n<\omega\}$. In $\mathfrak C_\omega$, observe that for each $m<\omega$ there is a subsequence $(C_{n_k})$  and a sequence of points $(x_k)\in \prod \{C_{n_k}:k<\omega\}$ such that $(x_k)$ converges to a point of $C_m$ and  $\overline{\bigcup \{C_{n_k}:k<\omega\}}$ is non-compact. 


\begin{ul}Let $X=\bigcup\{A_n:n<\omega\}$ be the union of a countable sequence of disjoint continua.  If  the closure $\overline{\bigcup \{A_{n_k}:k<\omega\}}$ is compact for every subsequence $(A_{n_k})$ such that $c(X)\cap \prod\{A_{n_k}:k<\omega\}\neq\varnothing$, 
 then the decomposition  $\widetilde{X}:=\{A_n:n<\omega\}$ is zero-dimensional. \end{ul}

\

\vspace{-7mm}

\begin{proof}It suffices to show $\widetilde{X}$ is regular because every countable regular space has dimension zero. (\textit{Hint:} Every  regular Lindel\"{o}f space is normal so Urysohn's Lemma can be applied.)

To that end, let $m<\omega$ and let $C$ be a closed subset of $X$ that misses $A_m$ and is a union of constituents; $C=\bigcup \{A_n:n\in I\}$ for some non-empty $I\subseteq \omega\setminus \{m\}$.  Assume that $m=0$. We find disjoint $X$-open sets $U_0$ and $U_1$, each of which is a union of continua from $\{A_n:n<\omega\}$ and such that $A_0\subseteq U_0$ and $C\subseteq U_1$.  

Recursively define two sequences of open sets $(U^n_0)$ and $(U^n_1)$ as follows.

\underline{Step $0$}: There exists $\epsilon_0>0$ such that  $$d(A_n,A_0)+d(A_n,C)\geq 4\epsilon_0$$  for all  $n<\omega$.\footnote{Here $d(A,B)=\inf \{d(x,y):x\in A\text{ and }y\in B\}$ for $A,B\subseteq X$.}  Otherwise, there is a sequence of arcs $(A_{n_k})$  such that $d(A_{n_k},A_0)+d(A_{n_k},C)\to 0$ as $k\to\infty$. There exists $x\in A_0$ and a sequence of points $x_j\in \bigcup \{A_{n_k}:k<\omega\}$ such that $x_j\to x$ as $j\to\infty$. Eventually $n_k\neq 0$ because $d(A_0,C)>0$. Thus $d(A_{n_k},A_0)>0$ for sufficiently large $k$, so that $$\{k<\omega:(\exists j<\omega)(x_j\in A_{n_k})\}$$ is infinite.  By hypothesis, a subsequence of $(A_{n_k})$ converges to a point $K\in K(X)$ (consult Proposition 5),  which is necessarily a continuum.  Then $x\in K$ and $d(K,C)=0$. As $d(A_0,C)>0$, this means $K\cap A_l\neq\varnothing$ for some $l\neq 0$, contradicting  Sierp\-i\'n\-ski's Theorem. 

 Put $A_{-1}=C$, $\epsilon_{-1}=\epsilon_0$, $N^0_0=\{0\}$, $N^0 _1=\{-1\}$, $$U^0_0=B(A_0,2\epsilon_0)\text{,\; and\; } U^0_1=B(A_{-1},2\epsilon_{-1}).$$This completes the base step.

\underline{Step $n$}: Suppose  $N^{n-1}_i\subseteq \omega\cup \{-1\}$ and open sets $U^{n-1}_i\subseteq X$ have been defined for each $i<2$, and that numbers $\epsilon_k>0$ ($k\in N^{n-1}_0\cup N^{n-1}_1$) have been chosen, so that:
\begin{equation}
\begin{aligned}
  &\textstyle{B\big(A_k,\epsilon_k[3-\sum_{j=0}^{n-1}2^{-j}]\big)\subseteq U^{n-1} _i\text{ for each }k\in N^{n-1}_i\text{; and}}\\[.5em]
  &A_k\cap U^{n-1}_i= \varnothing\text{ or }A_k\cap U^{n-1}_{1-i}= \varnothing\text{ for each }k<\omega.
\end{aligned}
\end{equation}

Let $k^*<\omega\setminus (N^{n-1}_0\cup N^{n-1}_1)$ be least such that $A_{k^*}\cap U^{n-1}_i\neq\varnothing$ for some (unique) $i<2$ (if there is no such $k^*$, then end the recursion and put $U^l_i=U^{n-1}_i$ and $N^l_i=N^{n-1}_i$ for each $l\geq n$ and $i<2$).
 
\begin{figure}[h]
  \centering
  \includegraphics[scale=.33]{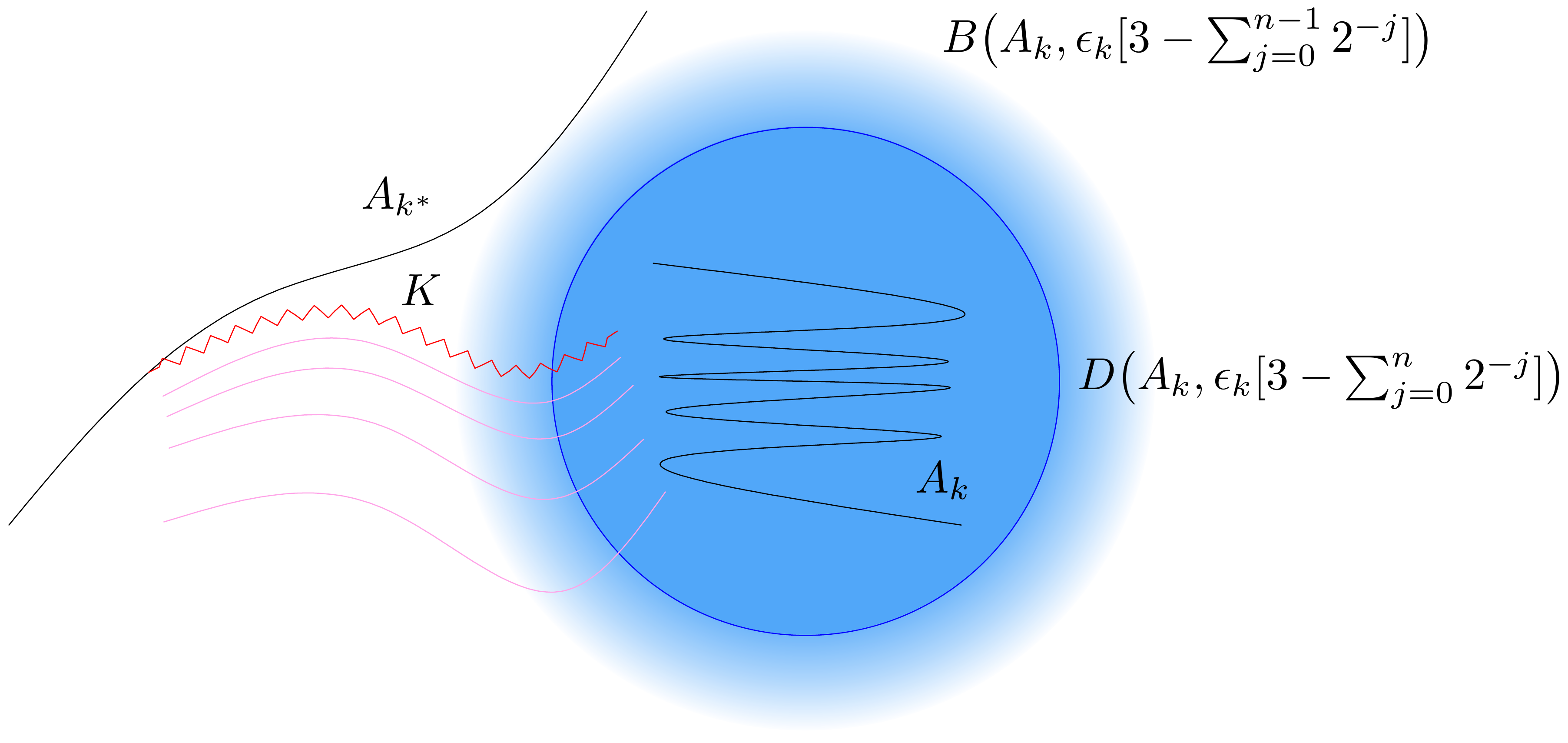} 
   \caption{Step $n$}
\end{figure}

There exists $\epsilon_{k^*}$ such that no constituent within $2\epsilon_{k^*}$ of $A_{k^*}$ also meets  
$$U^n_{1-i}:=\bigcup \big\{B\big(A_k,\epsilon_k\textstyle{[3-\sum_{j=0}^{n}2^{-j}]}\big):k\in N^{n-1}_{1-i}\big\}.$$ 
Otherwise, there exists $k\in N^{n-1}_{1-i}$ such that continua meeting $B\big(A_k,\epsilon_k[3-\sum_{j=0}^{n}2^{-j}]\big)$ get arbitrarily close to $A_{k^*}$. As in the base step,  a subsequence of continua  converges to a continuum $K\in K(X)$ with $$K\cap A_{k^*}\neq\varnothing\text{\;\; and \;\;} K\cap D\big(A_k,\epsilon_k[3-\textstyle{\sum} _{j=0}^{n}2^{-j}]\big)\neq\varnothing.$$ By (6.1) we have $A_{k^*}\cap D\big(A_k,\epsilon_k\textstyle{[3-\sum_{j=0}^{n}2^{-j}]}\big)=\varnothing$. Thus $K\cap A_l\neq\varnothing$ for some $l\neq k^*$, contradicting Sierp\-i\'n\-ski's Theorem.

Put $U^{n}_i:=U^{n-1}_i\cup B(A_{k^*},2\epsilon_{k^*})$, $N^n_i:=N^{n-1}_i\cup \{k^*\}$, and $N^n_{1-i}:=N^{n-1}_{1-i}$. The conditions in (6.1) are now  satisfied with all instances of $n-1$ replaced by $n$. This completes the recursive step. Let
\begin{align*}N_0&=\bigcup_{n<\omega} N^n_0; \\
N_1&=\Big(I\cup \bigcup_{n<\omega} N^n_1\Big)\setminus \{-1\};\text{ and} \\
U_i&=\bigcup_{k\in N_i}B(A_k,\epsilon_k)\text{ for each $i<2$}.
\end{align*}
Obviously $U_0$ and $U_1$ are open sets, $A_0\subseteq U_0$, and $C\subseteq U_1$. By construction we also have $U_0\cap U_1=\varnothing$, and $U_i$ is the union of  $\{A_k:k\in N_i\}$ for  $i<2$. \end{proof}



\section{Proof of Theorem II (Sufficiency)}
\noindent Suppose $f$ is  recurrent \& arc-complete. 

Let $\{d_i:i<\omega\}$ be a dense subset of $X$, and assume $\diam(X)>2$.

\begin{ucl} For each $i\in \{1,2,3,...\}$ there is a sequence $(A_n)$ of disjoint continua in $X\setminus B(d_i,\frac{1}{i+1})$ such that  
\begin{enumerate}[label=\textnormal{(\roman*)}]

\item $X\setminus B(d_i,\frac{1}{i})\subseteq X_i:=\bigcup \{A_n:n<\omega\};$ 

\item if $n\neq 0$ then $A_n$ is an arc; and

\item $\overline{X_i}=X_i$.\end{enumerate}\end{ucl}

\begin{proof}[Proof of Claim 7.1]

\
\smallskip

 \textit{Case 1:} $\ell=[0,\infty)$.
 \smallskip

By recurrence of $f$ there is a sequence increasing sequence $(r_k)\in [0,\infty)^\omega$ such that $f(r_k)\in B(d_i,\frac{1}{i+1})$ and $r_k\to \infty$ as $k\to\infty$. For each $z\in X\setminus B(d_i,\frac{1}{i})$ there is a unique $k<\omega$ such that $f^{-1}(z)\in (r_k,r_{k+1})$; let $M_z$ be the component of $f^{-1}(z)$ in $(r_k,r_{k+1})\setminus f^{-1}[B(d_i,\frac{1}{i+1})]$.   Each $M_z$ is a non-degenerate closed and bounded interval, and $\{M_z:z\in X\setminus B(d_i,\frac{1}{i})\}$ is countable.  Let $\{A_n:n<\omega\}$ be an enumeration of $\{f[M_z]:z\in X\setminus B(d_i,\frac{1}{i})\}$.

Properties (i) and (ii) are clear.  We  need to prove $\overline{X_i}\subseteq X_i$ for (iii). Well, let $y\in \overline{X_i}$. There is a sub-sequence of arcs $(A_{n_k})$ such that $y\in \overline{\bigcup \{A_{n_k}:k<\omega\}}$.  By the arc-complete  property,  a subsequence of $(A_{n_k})$ converges to a continuum $K\in K(X\setminus B(d_i,\frac{1}{i+1}))$.  Then $y\in K$ and $K\cap X\setminus B(d_i,\frac{1}{i})\neq\varnothing$.   Let $z\in K\cap X\setminus B(d_i,\frac{1}{i})$. By Proposition 4 and maximality of $M_z$ we have $y\in K\subseteq f[M_z]\subseteq X_i$. 

 \smallskip
\textit{Case 2:}  $\ell=(-\infty,\infty)$. 
 \smallskip
 
Assume $f$ is positively recurrent.  By the proof of Case 1 we may further assume that $f(-\infty,r]\cap B(d_i,\frac{1}{i+1})=\varnothing$ for some $r\in (-\infty,\infty)$. Let $$s=\sup\big\{r\in (-\infty,\infty):f(-\infty,r]\subseteq X\setminus B(d_i,\textstyle{\frac{1}{i+1}})\big\},$$ and put $R=f(-\infty,s]$.  As in Case 1 there is a sequence $(B_m)$ of disjoint arcs in $X\setminus B(d_i,\frac{1}{i+1})$ such that $$f [s,\infty)\setminus B(d_i,\textstyle{\frac{1}{i}})\subseteq \bigcup\{B_m:m<\omega\}$$ and $B_m\setminus B(d_i,\frac{1}{i})\neq\varnothing$ for each $m<\omega$.

There is at most one $m<\omega$ such that $\overline R \cap B_m\neq\varnothing$. For suppose otherwise that $\overline R$ intersects more than one of these arcs; without loss of generality $\overline R \cap B_0\neq\varnothing$ and $\overline R \cap B_1\neq\varnothing$.  Then there are  sequences $(r_n),(s_n)\in (-\infty,s]^\omega$ such that $s_{n+1}<r_{n}<s_n$ for every $n<\omega$, $r_n\to -\infty$ as $n\to\infty$, $(f(r_n))$ converges to a point in $B_{0}$, and $(f(s_n))$ converges to a point in $B_{1}$. Let $$X':=B_0\cup B_1\cup\bigcup \{f[r_n,s_n]:n<\omega\}.$$  By Lemma 1 there is an $X'$-clopen set $C$ such that $B_0\subseteq C$ and $C\cap B_1=\varnothing$. There exists $n<\omega$ such that $f(r_n)\in C$ and $f(s_n)\in X'\setminus C$, contradicting the fact that $f[r_n,s_n]$ is connected. 

If there exists such an $m$, call it $m^*$ and let $A_0=\overline R \cup B_{m^*}$, $A_n=B_{n-1}$ for $1\leq n \leq m^*$, and $A_n=B_{n}$ for $n>m^*$. Otherwise set $A_0=\overline R$ and $A_n=B_{n-1}$ for each $n\geq 1$. The sequence $(A_n)$ is as desired.\end{proof}


Each $\widetilde{X_i}$ (the set of continua components of $X_i$)   is zero-dimensional by Lemma 1.  $\widetilde{X_i}$ is also separable and metrizable, and thus has a basis of clopen sets $\{C_{\langle i,j\rangle}:j<\omega\}$ (cf.  \S46 V Theorem 3 \cite{kur}). For each $i<\omega$ let $\varphi_i:X_i\to \widetilde{X_i}$ be the canonical epimorphism.  

Since $X_i$ is closed in $X$, Urysohn's Lemma provides for each $\langle i,j\rangle\in \omega^2$  a mapping $f_{\langle i,j\rangle}:X\to [0,1]$ such that $$f_{\langle i,j\rangle}\big[\varphi_i ^{-1}[C_{\langle i,j\rangle}]\big]=0\text{\; and \;}f_{\langle i,j\rangle}\big[X_i\setminus \varphi_i ^{-1}[C_{\langle i,j\rangle}]\big]=1.$$
Let  $h:X\hookrightarrow [0,1]^\omega$ be a homeomorphic embedding of $X$ into the Hilbert cube such that for every $\langle i,j\rangle\in \omega^2$ there exists $n<\omega$ such that $\pi_n\circ h=f_{\langle i,j\rangle}$. 

Then $Y:=\overline{\overline{h[X]}}$ is a metrizable continuum in which $X$ is densely embedded, and 
\begin{equation}
\overline{\overline{\varphi_i ^{-1}[C_{\langle i,j\rangle}]}}\cap \overline{\overline{X_i\setminus \varphi_i ^{-1}[C_{\langle i,j\rangle}]}}=\varnothing\text{ for each }\langle i,j\rangle\in \omega^2.\end{equation}

\begin{ucl}$X$ is a composant of $Y$.\end{ucl} 
\begin{proof}[Proof of Claim 7.2]
It suffices to show $X$ contains every proper subcontinuum of $Y$ that meets   $X$. Well, suppose $K$ is a compact proper subset of $Y$ that contains points $x\in X$ and $y\in Y\setminus X$.  There exists $i<\omega$ such that $\overline{\overline{B(d_i,\frac{1}{i})}}\cap K=\varnothing$, so that $K\subseteq \overline{\overline {X_i}}$. Denote by $\{A_n:n<\omega\}$ the set of constituents of $X_i$. 

Let $m<\omega$ be the unique integer such that $x\in A_m$.  There is a sequence $(x_k)\in (X_i\setminus A_m)^\omega$ such that $x_k\to y$ as $k\to\infty$.  Let $n_k<\omega$ such that $x_k\in A_{n_k}$.  By arc-completeness $\bigcup\{A_{n_k}:k<\omega\}$ is closed in $X$. So there exists $j<\omega$ such that $A_m\in C_{\langle i,j\rangle}$ and  $C_{\langle i,j\rangle}\cap \{A_{n_k}:k<\omega\}=\varnothing$.  In addition to (7.1) we have  
\begin{align*}&x\in \overline{\overline{\varphi_i ^{-1}[C_{\langle i,j\rangle}]}}; \\
&p\in \overline{\overline{X_i\setminus \varphi_i ^{-1}[C_{\langle i,j\rangle}]}};\text{ and} \\
&K\subseteq \overline{\overline{X_i}}=\overline{\overline{\varphi_i ^{-1}[C_{\langle i,j\rangle}]}}\cup \overline{\overline{X_i\setminus \varphi_i ^{-1}[C_{\langle i,j\rangle}]}}.\end{align*}
Therefore  $K$ is not connected. \end{proof}

This completes our proof of Theorem II.

\section{Dimension-preserving compactifications}



Throughout this section, assume $X$ is a connected separable metric space. 

By \textit{a compactification of $X$} we shall mean a compact metrizable space in which $X$ is densely embedded. If   $\xi X$  and $\gamma X$ are two compactifications of $X$, then write $\xi X\geq \gamma X$ if  there is a continuous surjection $\hat f:\xi X\to \gamma X$ such that $\hat f\restriction X$ is the inclusion $X\hookrightarrow \gamma X$. More precisely, $\hat f\restriction \xi[X]=\gamma\circ \xi^{-1}$, where $\xi:X\hookrightarrow \xi X$ and $\gamma:X\hookrightarrow \gamma X$ are dense homeomorphic embeddings.

\begin{ul} For every compactification $\gamma X$  there is a compactification $\xi X\geq \gamma X$ such that $\dim(\xi X)=\dim(X)$.\end{ul}

\begin{proof}Assume $\gamma X\subseteq [0,1]^\omega$, and let $\pi_n:[0,1]^\omega\to [0,1]$ be the $n$-th coordinate projections. According to 1.7.C in  \cite{eng3} (also \cite{eng2}), there is a compactification $\xi X$   such that $\dim(\xi X)=\dim(X)$ and each mapping $f_n:=\pi_n\circ \gamma\circ \xi^{-1}:\xi[X]\to [0,1]$ continuously extends to $\xi X$. Let $\hat f_n:\xi X\to [0,1]$ be the extension of $f_n$, and define $\hat f:\xi X\to [0,1]^\omega$ by $\pi_n\circ \hat f=\hat f_n$.  Then $\hat f$ maps onto $\gamma X$ and witnesses $\xi X\geq \gamma X$.\end{proof}

\begin{ul}If $X$ is a composant of $\gamma X\leq \xi X$, then $X$ is a composant of $\xi X$.\end{ul}

\begin{proof}Let $\hat f$ witness $\xi X\geq \gamma X$.  Since $\hat f\restriction \xi[X]=\gamma\circ \xi^{-1}$ is a homeomorphism onto $\gamma[X]$, and $\gamma [X]$ is dense in $\gamma X$,  we have 
\begin{equation}\hat f^{-1}[ \gamma[X]]=\xi[X].\end{equation} For the remainder of the proof let us identify  $X$, $\gamma[X]$, and $\xi [X]$.  

Suppose $X$ is a composant of $\gamma X$.  Let $x\in X$ be such that $X$ is the composant of $x$ in $\gamma X$. Let $P$ be the composant of $x$ in $\xi X$. Apparently, $X\subseteq P$. On the other hand, if $z\in P\setminus X$, and $Z\supseteq \{x,z\}$ is a proper subcontinuum of $\xi X$, then by (9.1) the subcontinuum $\hat f[Z]$ is proper and violates maximality of $X$ in $\gamma X$.  Thus $P\subseteq X$.  Combining the two inclusions, we have $P=X$.\end{proof}




Lemmas 2 and 3 directly imply Theorem III (stated in Section 1).

\section{Examples in the plane}  
Figures 3 though 5  show seven non-homeomorphic indecomposable plane sets.  All  are one-to-one images of $[0,\infty)$, except for $X_3$, $X_5$ and $X_6$, which are one-to-one images of $(-\infty,\infty)$. Examples $X_0$, $X_3$, $X_4$ and $X_5$ are  composants; $X_1$, $X_2$ and $X_6$ are not.
\vspace{-.2cm}
\begin{figure}[H]
  \centering
  \includegraphics[scale=.45]{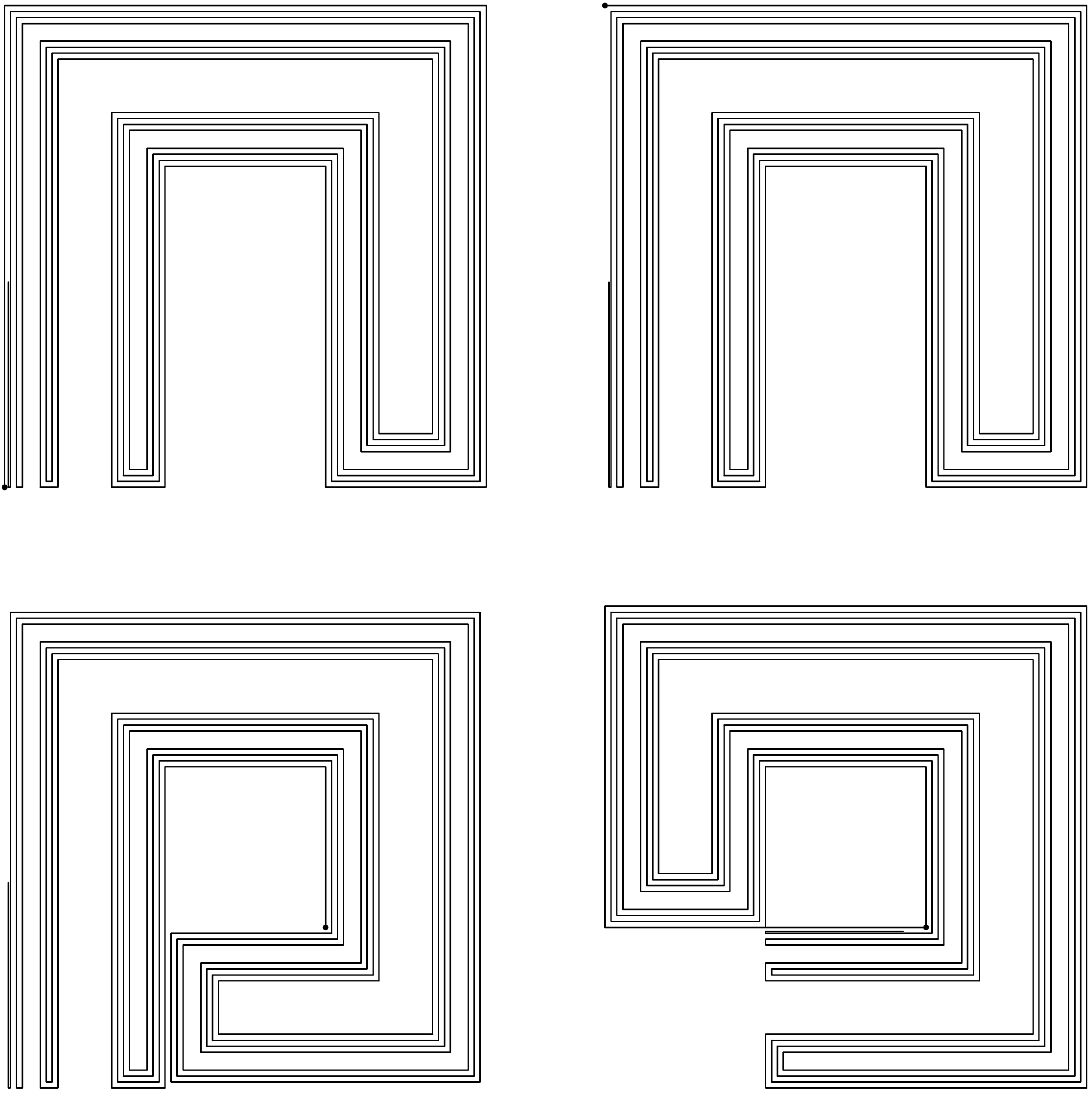} 
 \caption[]{$\begin{array}{lr} X_0 & \ X_1 \\ \\ X_2 & \ X_3 \end{array}$}
  \end{figure}


\begin{itemize} 
\item $X_0$ is the well-known visible composant of the bucket-handle continuum. It is a one-to-one recurrent image of $[0,\infty)$.

\item $X_1$ and $X_2$ are a one-to-one recurrent images of $[0,\infty)$ which fail to be  composants. $X_1$ is not arc-complete, as there is a sequence vertical arcs in $X_1$ whose endpoints limit to both $\langle 0,0\rangle\notin X_1$ and $\langle 0,1\rangle\in X_1$. $X_2$ is not arc-complete; consider the horizontal arcs which limit to both $\langle \frac{1}{3},\frac{1}{3}\rangle$ and $\langle \frac{2}{3},\frac{1}{3}\rangle$.  $X_2$ also has a disconnected proper quasi-component.  The two intervals $P:=\{\frac{1}{3}\}\times [0,\frac{1}{3}]$ and $Q:=\{\frac{2}{3}\}\times[\frac{1}{3},\frac{1}{2}]$ form a quasi-component of the proper closed subset $X_2\cap ([0,1]\times [0,\frac{1}{2}])$.  This contrasts with $X_1$, whose only proper quasi-components are arcs. In any compactification of $X_2$, $P$ and $Q$ must be joined by a proper subcontinuum that goes outside of $X_2$.

\item $X_3$ is a composant image of $(-\infty,\infty)$ which is recurrent but not bi-recurrent.  The invisible composants of the bucket-handle, as well as the composants of the solenoid, are bi-recurrent.

\begin{figure}[H]
\centering
  \includegraphics[scale=.55]{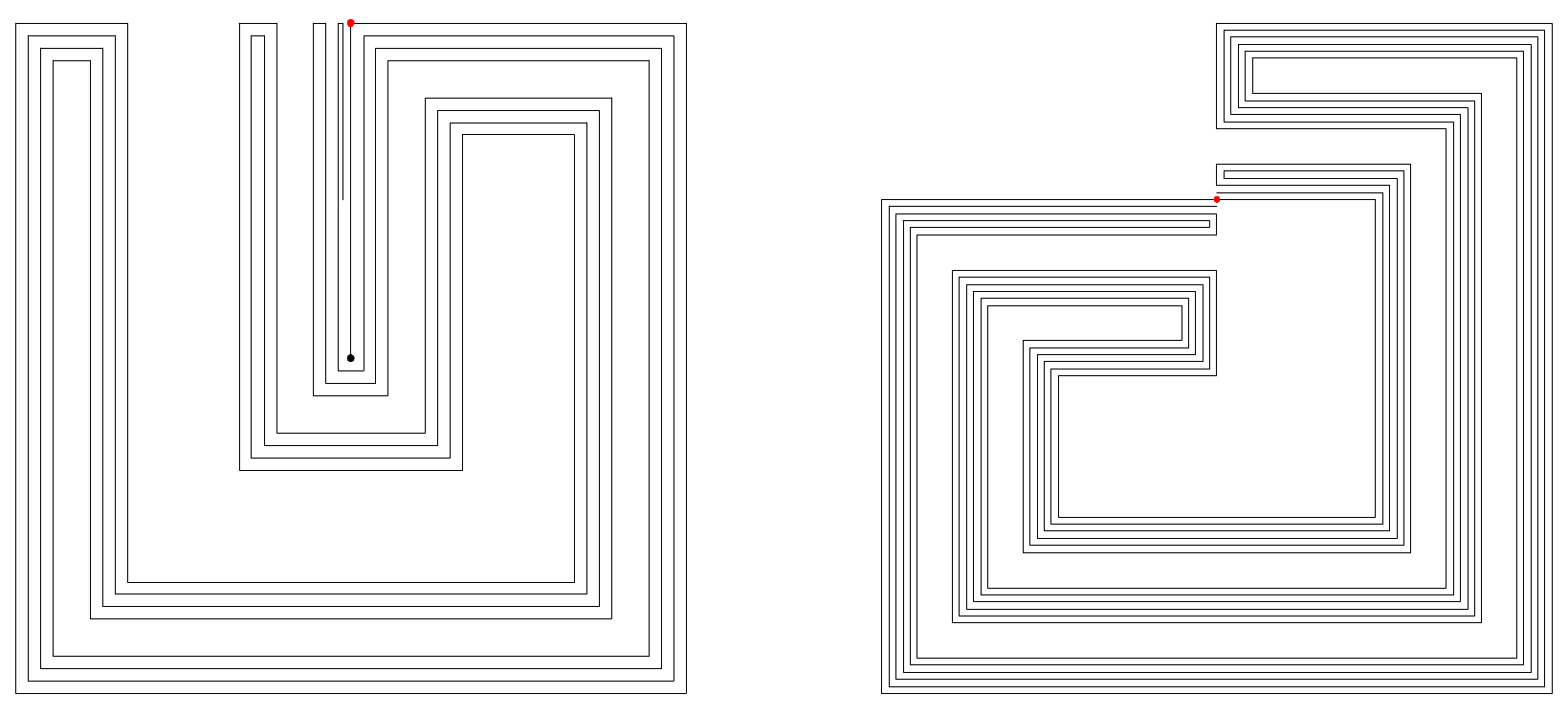}  
  \caption[Text excluding the matrix]{$X_4$\;\;\;\;\;$X_5$}
 \end{figure}

\item $X_4$ is a  one-to-one recurrent composant image of $[0,\infty)$.  $X_5$ and $X_6$ are one-to-one bi-recurrent images of  $(-\infty,\infty)$.\footnote{These examples derive from the quinary double bucket-handle continuum (the inverse limit of arcs with $N$-shaped bonding map).  That continuum has two accessible composants, each of which is a one-to-one image of $[0,\infty)$; $X_5$ is obtained by gluing together the endpoints of these two composants. $X_6$ is a one-to-one continuous image of $X_5$.} $X_5$ is a composant;  $X_6$ is not.  Observe that $X_4$, $X_5$, and $X_6$  have no $\mathbb Q \times (-1,1)$-neighborhoods at their red points. 
\begin{figure}[H]
\centering
  \includegraphics[scale=.37]{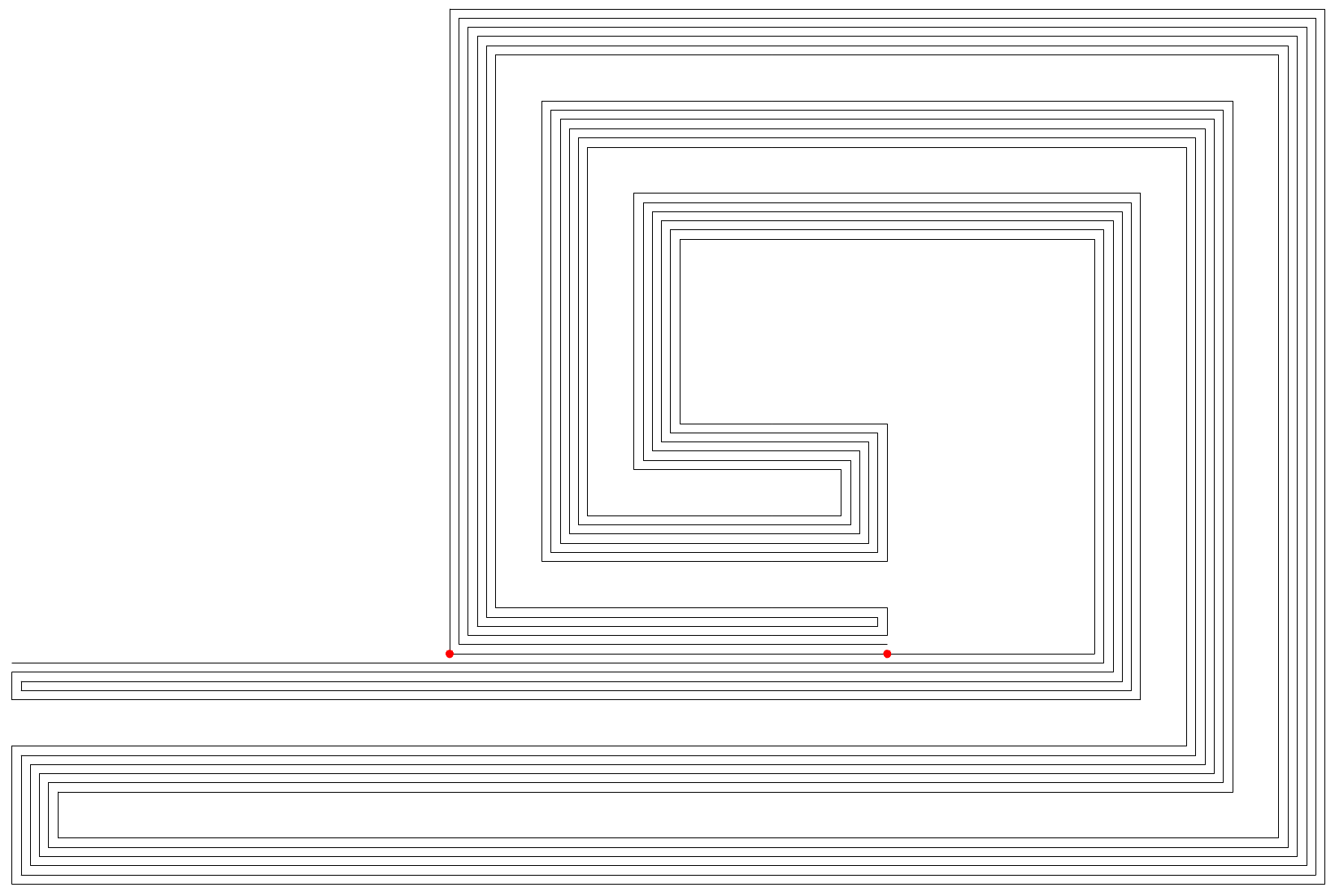}  
  \caption[Text excluding the matrix]{$X_6$}
 \end{figure}
\end{itemize} 







\section{Questions}


 
\begin{uq}Does every one-to-one  composant image of $[0,\infty)$ embed into the plane? \end{uq}


\begin{uq}If $X$ is a one-to-one recurrent image of $[0,\infty)$, and $Y$ is a continuum of which $X$ is a composant, then is $Y$ necessarily chainable?\end{uq}

Compact one-to-one images of $[0,\infty)$ embed into the plane \cite{nad3}. So by Theorem I and the fact that chainable continua are planar, a positive answer to Question 2 implies a positive answer to Question 1.

\begin{uq}Is every linear indecomposable composant equal to a composant of  a continuum each of whose composants is linear? \end{uq}





The final two questions are motivated by a result of F. Burton Jones \cite{jon,jon2}: Every locally connected one-to-one plane image of the (half-)line is locally compact.

\begin{uq}Is every recurrent one-to-one  plane image of $[0,\infty)$ indecomposable?\end{uq}

  \begin{uq}Is every bi-recurrent one-to-one  plane image of $(-\infty,\infty)$ indecomposable?\end{uq}

 Without restricting to the plane, the answers to  Questions 5 and 6 are \text{\tt no}.  There is a one-to-one image of $[0,\infty)$ which is both locally connected and dense in Euclidean 3-space \cite{jon}. By similar methods, one obtains a locally connected  bi-recurrent one-to-one image of $(-\infty,\infty)$.














\end{document}